\documentclass[10pt]{article}
\usepackage{graphicx} 
\usepackage{amsmath}
\usepackage{amssymb}
\usepackage{amsthm}
\usepackage{mathtools}
\usepackage{color}
\usepackage{hyperref}
\usepackage{authblk}
\usepackage[paper=a4paper,total={19cm, 26cm}]{geometry}
\usepackage{faktor}
\usepackage{framed}
\newtheorem*{remark}{Remark}
\newtheorem*{problem}{Problem}

\newtheorem*{proposition}{Proposition}
\newtheorem*{definition}{Definition}

\newcommand{\Pzero}{$\mathrm{P0}$}
\newcommand{\Pone}{$\mathrm{P1}$}
\newcommand{\PoneA}{$\mathrm{P1}_\mathrm{A}$}
\newcommand{\PoneB}{$\mathrm{P1}_\mathrm{B}$}
\newcommand{\PoneC}{$\mathrm{P1}_\mathrm{C}$}
\newcommand{\Ptwo}{$\mathrm{P2}$}
\newcommand{\Ptwoh}{$\mathrm{P2}^h$}

\title{A regularized method for quadratic optimization problems\\ with finite-dimensional
degeneracy}

\author[1]{\small Cristian Guillermo Gebhardt\thanks{Corresponding author.\textit{E-mail-address}:\href{mailto:cristian.gebhardt@uib.no}{cristian.gebhardt@uib.no}}}
\author[2,3]{\small Ignacio Romero}

\affil[1]{\small Geophysical Institute and Bergen Offshore Wind Centre (BOW), University of Bergen, Allégaten 70, 5007 Bergen, 
Norway}
\affil[2]{\small Deparment of Mechanical Engineering, Universidad Politécnica de Madrid, Jos\'{e} Guti\'{e}rrez Abascal 2, 28006 Madrid, Spain}
\affil[3]{IMDEA Materials Institute, Eric Kandel 2, 28096 Getafe, Spain}

\date{}

\begin{document}

\maketitle
\vspace{-12mm}
\begin{abstract}
\noindent We propose and analyze a perturbative regularization method to approximate quadratic optimization problems with finite‑dimensional degeneracy. The original problem is first approximated by a regularized problem depending on a small positive parameter, and then discretized using the finite element method. The resulting families of continuous and discrete functionals $\Gamma$‑converge to the functional of the original problem and the corresponding minimizers converge as well. Our method generalizes the approach proposed in \cite{Kaleem2026} for numerically approximating pure Neumann problems, which represents the cornerstone of a sparsity-preserving, numerically efficient alternative to the methods developed in \cite{Bochev2005}, \cite{Ivanov2019} and \cite{Roccia2024}.

\vspace{2mm}

\vspace{2mm}
\end{abstract}

\section{Introduction}

In this letter, we study certain approximations intended to regularize quadratic optimization problems with finite-dimensional degeneracy. To this end, let us first introduce the Hilbert spaces $L$ and $H$ such that the triplet $(H, L, H')$ -- where $H'$ denotes the topological dual space of $H$ -- is rigged, i.e., $H \subset L \subset H'$, and $L \cong L'$ acts as the pivot through its identification with its dual. For the space $L$, we denote the inner product by $(\cdot,\cdot)_L$ and the norm by $\|\cdot\|_L$. 
In addition, let $K \subset H$ be the kernel of the bilinear form of the problem so that $H = K \oplus K^\perp$ together with the corresponding projectors $\Pi^\parallel : H \to K$ and $\Pi^\perp : H \to K^\perp$. 
Given $u$ and $v$ in $H$, we say $u \sim v$ (are $K$-equivalent) if $u - v \in K$. We then identify $H/K$ as the quotient space of $H$ modulo $K$. The equivalence class of $u$, denoted by $[u]$, is the set of all $v\in H$ that are $K$-equivalent to $u$. Since $K$ is a closed subspace of $H$, $H/K$ is a normed space with norm
\begin{equation*}
    \| [u] \|_{H/K} = \inf_{v \in [u]\in H/K} \| v \|_{H}.
\end{equation*}
Moreover, the orthogonal projector $\Pi^\perp:H \to K^\perp$ induces a natural isometric isomorphism between the quotient space $H/K$ and $K^\perp$. Specifically, we can identify $K^\perp\cong H/K$ via the map $[u] \in H/K \mapsto \Pi^\perp (u) \in K^\perp$.

\noindent Now, we can proceed with the formal definition of the problem that motivates the present study.

\begin{framed}
\begin{problem}[\Pzero~-- original problem]
Given: \textit{i}) $a(\cdot, \cdot):H\times H\to\mathbb{R}$, a symmetric bilinear form such that $a(u, v) \leq C_a \| u \|_H\| v \|_H ~\forall~u,v \in H$, $a(u, u) \geq c_a \|u\|^2_H ~\forall~u \in H / K$ and $a(u, v) = a(v, u) = 0 ~\forall~u \in K$ and $~\forall~v \in H$; and, \textit{ii}) $\ell:H\to\mathbb{R}$, a linear form such that $\ell(u) \leq \|u \|_H \| \ell \|_{H'}~\forall~u \in H$ and $\ell(u) = 0~\forall~u \in K$. Let
\begin{equation*}
    V(u) = \frac12 a(u, u) - \ell(u)
\end{equation*}
be a functional on $H$. Then, find
\begin{equation*}
    [u] = \arg \left \lbrace \inf_{[\bar{u}] \in H/K} V(\bar{u}) \right \rbrace.
\end{equation*}
\end{problem}
\end{framed}

\noindent The main issue with \Pzero, even when it is well-defined on the quotient space $H/K$, is that it does not provide a unique representative in $H$. Indeed, each equivalence class $[u]\in H/K$ contains infinitely many functions in $H$, differing by arbitrary components in $K$. To select a unique function in $K^\perp$, we need then to impose $\Pi^\parallel(u) = 0$.

\begin{framed}
\begin{proposition}
    Let $[u]\in H/K$ be the solution of \Pzero. The representative $u\in[u]$ satisfying $\Pi^\parallel(u)=0$ is the unique element of the equivalence class $[u]$ with minimal $L$-norm. This representative belongs 
    to~$K^\perp$.
\end{proposition}
\end{framed}

\begin{proof}
    Let $[u] \in H/K$ be the solution of \Pzero~ and $u_0\in[u]$ be the representative satisfying $\Pi_K^\parallel(u_0) = 0$ so that $u_0 = \Pi^\perp(u_0)\in K^\perp$. For any other representative of $u \in[u]$, we can write $u = \Pi^\parallel(u)+\Pi^\perp(u)$. Since all element of the equivalence class differ by elements of $K$, we have $\Pi^\perp(u) = \Pi^\perp(u_0) = u_0$. Using the orthogonality of the decomposition $H = K\oplus K^\perp$ (and the fact that the decomposition is also orthogonal in $L$), we obtain
    \begin{equation*}
        \| u \|^2_L = \| \Pi^\parallel(u)\|^2_L + \| \Pi^\perp(u)\|^2_L \geq \| \Pi^\perp(u)\|^2_L = \|u _0\|^2_L.
    \end{equation*}
    taking square roots yields $\|u\|_L \geq \| u_0\|_L$. 
    Uniqueness follows from the strict orthogonality of the decomposition.
\end{proof}
\noindent The reminder is as follows. Section 2 introduces equivalent reformulations of \Pzero. Section 3 presents and analyzes the perturbative $L$-regularization in the continuous setting. Section 4 parallels the previous section, but in the discrete setting.

\section{Equivalent formulations}

Next, we define three well-posed problems that enforce $\Pi^\parallel(u) = 0$ in different ways. The first formulation extends the functional $V$ from $K^\perp$ to $H$ by assigning the value $+\infty$ outside $K^\perp$.

\begin{framed}
\begin{problem}[\PoneA -- extended formulation]
Let 
\begin{equation*}
    I_{K^\perp}(u) :=
    \left\lbrace
    \begin{array}{ll}
     0,    & \text{if}\quad u \in K^\perp\\
     +\infty, & \text{otherwise},
    \end{array}
    \right.
\end{equation*}
be the indicator function of $K^\perp$ and let $u \in H$. Then, find
\begin{equation*}
    u =
    \arg \left \lbrace \inf_{\bar{u} \in H} \left(V(\bar{u}) + I_{K^\perp}(\bar{u})\right)\right \rbrace .
\end{equation*}    
\end{problem}
\end{framed}

\noindent The second formulation enforces the constraint $\Pi^\parallel(u) = 0$ by means of Lagrange multipliers \cite{Bochev2005,Roccia2024}.

\begin{framed}
\begin{problem}[\PoneB~-- saddle-point formulation]
Let 
\begin{equation*}
    V_\lambda(u) := \sup_{\lambda \in L} \left\lbrace V(u) + \left( \lambda, \Pi_K^\parallel (u) \right)_L \right\rbrace
\end{equation*}
be the augmented functional to be minimized, where $u \in H$. Then, find
\begin{equation*}
    u_\lambda = \arg \left \lbrace \inf_{\bar{u} \in H} V_\lambda(\bar{u}) \right \rbrace .
\end{equation*}
\end{problem}
\end{framed}

\noindent The third and last formulation approximately enforces the constraint $\Pi_K^\parallel(u) = 0$ by a penalty term \cite{Ivanov2019,Roccia2024}.

\begin{framed}
    \begin{problem}[\PoneC~-- penalty formulation]
Let $\epsilon>0$ and
\begin{equation*}
    V_\epsilon(u) := V(u) + \frac{1}{2\epsilon}\| \Pi_K^\parallel (u)\|^2_L
\end{equation*}
be the functional to be minimized and $u \in H$. Then, find
\begin{equation*}
    u_\epsilon = \arg \left \lbrace \inf_{\bar{u} \in H} V_\epsilon(\bar{u}) \right \rbrace .
\end{equation*}
\end{problem}    
\end{framed}

\begin{framed}
\begin{proposition}
    \PoneA, \PoneB~ and \PoneC~ are equivalent. Consequently, we shall refer to any of them as $\mathrm{P1}$, emphasizing that their unique solution is precisely the representative of the corresponding equivalence class in $H/K$ with minimal $L$-norm.  
\end{proposition}
\end{framed}

\begin{proof}
The problems are equivalents if only if their solutions all lie in $K^\perp$ and simultaneously every solution in $K^\perp$ satisfies the same optimality condition for all problems.
For \PoneB, we need to show that the Lagrange multiplier $\lambda$ is always $L$-orthogonal to the constraint and for \PoneC, we need to show that the solution is independent from $\epsilon$. At the same time, we need to verify that the solutions of \PoneB~ and \PoneC~ identically satisfy the optimality condition required for \PoneA. Below we explicitly verify these points and show that all three formulations lead to the same solution in $K^\perp$.

\noindent By construction, the solution $u \in H$ of \PoneA~ satisfies
\begin{equation*}
    a(u, v) = \ell(v)~ \forall~ v \in K^\perp,
    \quad
    a(u, v_K) = \ell(v_K) = 0~ \forall~ v_K \in K,   
\end{equation*}
which means that $u \in K^\perp$.
The solution $u_\lambda$ of \PoneB~ satisfies
\begin{equation*}
a(u_\lambda, v) + \left(\lambda, \Pi^\parallel(v)\right)_L = \ell(v)~ \forall~ v \in H,
\quad
\left(\Pi^\parallel(u_\lambda), \mu\right)_L =0~ \forall~ \mu \in L.
\end{equation*}
The second condition gives $\Pi^\parallel(u_\lambda) = 0$ and therefore $u_\lambda \in K^\perp$. Now, chose $v = v_K \in K$ in the first equation. Since $\Pi^\parallel(v_K) = v_K$, we obtain
\begin{equation*}
    (\lambda, v_K)_L = 0~\forall~v_K \in K.
\end{equation*}
Thus the constraint term vanishes on $K^\perp$, and the first equation restricted to $K^\perp$ becomes
\begin{equation*}
    a(u_\lambda, v) = \ell (v)~ \forall~ v \in K^\perp.
\end{equation*}
The solution of $u_\epsilon$ of \PoneC~ satisfies
\begin{equation*}
    a(u_\epsilon, v)+\frac{1}{\epsilon}\left(\Pi^\parallel(u_\epsilon), \Pi^\parallel(v)\right)_L = \ell(v)~ \forall~ v \in H. 
\end{equation*}
Taking $v = v_K \in K$, and using that $\Pi^\parallel(v_K) = v_K$, we obtain
\begin{equation*}
    \frac{1}{\epsilon}\left(\Pi^\parallel(u_\epsilon), \Pi^\parallel(v_K)\right)_L = 0 ~\forall~\epsilon > 0,~\forall~v_K\in K.
\end{equation*}
This implies $\Pi^\parallel(u_\epsilon) = 0$, and hence $u_\epsilon \in K^\perp$.
For $u_\epsilon \in K^\perp$, the penalty term vanishes and the optimality condition becomes
\begin{equation*}
    a(u_\epsilon, v) = \ell(v)~ \forall~ v \in K^\perp. 
\end{equation*}
Thus, $u_\epsilon$ is independent from $\epsilon$.
The equivalence is finally established by considering that any two solutions $u_1$ and $u_2$ in $K^\perp$ satisfy
\begin{equation*}
a(u_1, v) - \ell(v) -\left(a(u_2, v) -\ell(v)\right) = 0~ \forall~ v \in K^\perp,
\end{equation*}
which implies
\begin{equation*}
a(u_1-u_2, v) = 0~ \forall~ v \in K^\perp.
\end{equation*}
Then, we conclude $u_1 = u_2$.
\end{proof}

\section{A perturbative $L$-regularization - continuous setting}

To approximate the constrained problem \Pone, we introduce a perturbative regularization in which the constraint is enforced approximately through a $L$-penalty. This leads to the family of unconstrained minimization problems \Ptwo, defined by adding a small, parameter‑dependent stabilization term to the original functional. The regularized problems are well posed for every $\eta>0$, and their solutions converge to the solution of \Pone~as $\eta\downarrow 0$ in a variational sense.

\begin{framed}
\begin{problem}[\Ptwo~-- perturbative regularization] Let  
\begin{equation*}
    V_\eta(u) := V(u) + \frac{\eta}{2}(u, u)_L 
\end{equation*}
be the regularized functional to be minimized, where $u \in H$, $\eta > 0$ is a small parameter, and $\ell(v_K) = 0~ \forall~ v_K \in K$ (consiwtency of $\ell$ for \Pone).
Then, find
\begin{equation*}
    u_\eta = \arg \left \lbrace \inf_{\bar{u} \in H} V_\eta(\bar{u}) \right \rbrace.
\end{equation*}
The minimizer $u_\eta$ equivalently satisfies the optimality condition
\begin{equation*}
    a(u_\eta, v) + \eta(u_\eta, v)_L = \ell(v)~ \forall~v \in H.
\end{equation*}
\end{problem}
\end{framed}

\begin{remark}
It is immediate that \Ptwo~is solvable and has a unique solution, since the regularized functional is strictly convex, continuous, and coercive on $H$ for every $\eta>0$.
\end{remark}

\begin{framed}
\begin{proposition}
    Let $[u] \in H/K$ be the equivalence class of solutions of \Pone, and let $u_\eta$ be the solution of \Ptwo. Then there exists a constant $C>0$, independent of $\eta$ such that
    \begin{equation*}
        \|u_\eta - u_0\|_H \leq C \eta \| \ell \|_{H'},
    \end{equation*}
    where $u_0$ is the element of $[u]$ with minimal $L$-norm, i.e., the solution of \Pone.
\end{proposition}
\end{framed}

\begin{proof}
From the optimality condition of \Ptwo, we have
\begin{equation*}
    a(u_\eta, v_K) + \eta(u_\eta, v_K)_L = \ell(v_K)~ \forall~ v_K \in K.  
\end{equation*}
Since $\ell(v_K)=0$ for all $v_K\in K$, this implies
\begin{equation*}
    (u_\eta, v_K)_L = 0~ \forall~ v_K \in K.  
\end{equation*}
Hence $\Pi^\parallel (u_\eta) = 0$ and therefore $u_\eta \in K^\perp$, i.e., $\Pi_K^\parallel(u_\eta-u_0) = 0$.

\noindent On the one hand, by setting $v = u_\eta -u_0$ and subtracting the optimality condition for \Ptwo~from that of \Pone, we obtain
\begin{equation*}
        a(u_\eta - u_0, u_\eta  - u_0) - \eta(u_\eta, u_\eta - u_0)_L = 0.
\end{equation*}
Then, by invoking the coercivity of $a(\cdot, \cdot)$ on $K^\perp$, we obtain
\begin{equation*}
    \|u_\eta - u_0\|_H \leq C \eta \|u_\eta\|_L.
\end{equation*}

\noindent On the other hand, by considering the optimality condition of \Ptwo~and setting $v = u_\eta$ combined with the continuity of $a(\cdot,\cdot)$ and $\ell(\cdot)$ on $H$, we obtain
\begin{equation*}
    \| u_\eta \|_H^2 \leq C\| \ell \|_{H'} \| u_\eta\|_H. 
\end{equation*}

\noindent By combining both results and using the continuous embedding $\|w\|_L \leq C\|w\|_H$ for all $w \in H$, we conclude
\begin{equation*}
    \| u_\eta -u_0 \|_H \leq C\eta \| \ell \|_{H'}.
\end{equation*}
\end{proof}

\noindent This result establishes the convergence of the solution $u_\eta$ of \Ptwo~to the solution $u_0$ of \Pone~as $\eta \downarrow 0$. However, strong convergence of minimizers does not -- in general -- imply convergence of the corresponding minimum problems. To investigate the asymptotic behavior of the functionals themselves, and to guarantee convergence of minimizers and minimum values in a variational sense, we employ the framework of $\Gamma$-convergence. This analysis is carried out below.   

\begin{framed}
\begin{proposition}
    The family of functionals associated with \Ptwo, denoted $\lbrace V_\eta \rbrace_{\eta\geq 0}$, $\Gamma$-converges to the functional of \Pone~in the strong topology as $\eta\downarrow 0$. Moreover, the corresponding minimum problems are convergent, i.e.,
    $
        \inf_{K^\perp} V_\eta(u_\eta) \to \inf_{K^\perp} V_0(u_0),
    $
    and the minimizers $u_\eta$ converge strongly in $K^\perp$ to the unique minimizer $u_0$ of \Pone.
\end{proposition}
\end{framed}

\begin{proof} To establish the convergence of the minimum problems, we need to examine three conditions: \textit{i}) equi-coercivity of the family $\lbrace V _\eta\rbrace_{\eta \geq 0}$; \textit{ii}) the \textit{liminf} inequality; and, \textit{iii}) the existence of a \textit{recovery} sequence. The last two conditions are sufficient to establish the $\Gamma$-convergence of $V_\eta$ to $V_0$, while equi-coercity ensures the convergence of the corresponding minimum problems and their minimizers.
Since the linear form satisfies $\ell(v_K) = 0$ for any $v_K \in K$, we take $K^\perp$ as the natural space within this proof. Nevertheless, \Ptwo~is well defined on the whole space $H$, which is essential for numerical implementations.

\noindent \textit{i}) Equi-coercivity:
Let $a_\eta(\cdot, \cdot)$ be defined as $a(\cdot, \cdot) + \eta (\cdot, \cdot)_L$. Then,  for every $u\in K^\perp$, 
\begin{equation*}
    a_\eta(u, u) = a(u, u)+\eta\|u\|^2_L \geq a(u, u) \geq c_a \|u\|^2_H,
\end{equation*}
and therefore, $a_\eta(\cdot, \cdot)$ is coercive on $K^\perp$ for any $\eta \geq 0$. Since the coercivity constant on $K^\perp$ is independent of $\eta$, every sublevel set $\lbrace u \in K^\perp: V_\eta(u) \leq t\rbrace$ contains a bounded subset of $K^\perp$. Hence the family $\lbrace V_\eta\rbrace_{\eta\geq 0}$ is equi-coercive.

\noindent \textit{ii}) The \textit{liminf} inequality:
Let $u \in K^\perp$ and let $u_\eta \in K^\perp$ be such that $u_\eta \rightharpoonup u \in K^\perp$. Then,
\begin{equation*}
    V_\eta(u_\eta) = V(u_\eta)+\frac{\eta}{2}\|u_\eta\|_L^2\geq V(u_\eta).
\end{equation*}
Given the properties of $a(\cdot,\cdot)$ and $\ell(\cdot)$ defined on $K^\perp$, $V(\cdot) = \frac12 a(\cdot,\cdot)-\ell(\cdot)$ is weakly lower semicontinuous on $K^\perp$, we obtain
\begin{equation*}
    \liminf_{\eta \downarrow 0} V_\eta(u_\eta) \geq \liminf_{\eta \downarrow 0} V(u_\eta) \geq V(u).
\end{equation*}
\textit{iii}) The existence of a \textit{recovery} sequence:
Let $u \in K^\perp$ and define $u_\eta \in K^\perp$ as the unique solution of
\begin{equation*}
    a_\eta(u_\eta,v) = a(u,v)~\forall~v\in K^\perp.
\end{equation*}
Recall that $a_\eta(\cdot,\cdot)$ is coercive on $K^\perp$ with a coercivity constant $c_a$ and continuous on $H$ with a continuity constant $C_a$ and both constants are independent of $\eta$. Now, by taking
$v = u_\eta \in K^\perp$, we can infer the uniform bound $\| u_\eta \|_H \leq C_a/c_a \|u\|_H$. Since bounded sequences in $H$ admit weakly convergent subsequences, and since the limit is unique, the whole sequence satisfies  $u_\eta \rightharpoonup u \in K^\perp$.
To prove that $\lbrace u_\eta\rbrace_{\eta \geq 0}$ is a recovery sequence, choose $v=u_\eta-u \in K^\perp$ and obtain
\begin{equation*}
a(u_\eta-u,u_\eta-u)+\eta\|u_\eta-u\|^2_L = -\eta(u,u_\eta-u)_L,
\end{equation*}
By Young's inequality,
\begin{equation*}
-\eta(u,u_\eta-u)_L \leq \frac{\eta}{2}\|u\|^2_L+\frac{\eta}{2}\|u_\eta-u\|^2_L.
\end{equation*}
Hence
\begin{equation*}
a(u_\eta-u,u_\eta-u)+\eta\|u_\eta-u\|^2_L \leq \frac{\eta}{2}\|u\|^2_L+\frac{\eta}{2}\|u_\eta-u\|^2_L.
\end{equation*}
Rearranging and using the coercivity of $a(\cdot,\cdot)$ on $K^\perp$,
\begin{equation*}
c_a\|u-u_\eta\|^2_H \leq a(u_\eta-u,u_\eta-u)+\frac{\eta}{2}\|u_\eta-u\|^2_L \leq \frac{\eta}{2}\|u\|^2_L.
\end{equation*}
Taking limits
\begin{equation*}
\lim_{\eta \downarrow 0} c_a\|u-u_\eta\|^2_H \leq \lim_{\eta \downarrow 0} \frac{\eta}{2}\|u\|^2_L = 0.
\end{equation*}
Hence $u_\eta \to u \in K^\perp$. Finally, we obtain
\begin{equation*}
    \lim_{\eta \downarrow 0} V_\eta(u_\eta) = \lim_{\eta \downarrow 0} \left(a(u_\eta, u_\eta)-\ell(u_\eta)+\frac{\eta}{2}\|u_\eta\|^2_L\right) = V(u),
\end{equation*}
because $u_\eta \to u \in K^\perp$.
Thus, $\lbrace u_\eta \rbrace_{\eta\geq 0}$ is a recovery sequence for $u$.
From \textit{ii}) and \textit{iii}), the $\Gamma$-convergence of $V_\eta$ as $\eta\downarrow 0$ is established. Finally, \textit{i}) and the $\Gamma$-convergence provide the convergence of the minimum problems and of their minimizers.
\end{proof}

\section{A perturbative $L$-regularization - discrete setting}

\begin{framed}
\begin{definition}[Admissible finite element spaces]
Let $H$ be a Hilbert space and let $\tilde{H}\subset H$ be a stronger regularity space. We say that a family of finite-dimensional subspaces $\lbrace H^h\rbrace^{h>0}$ such that $H^h \subset H$ is admissible if it satisfies:
\begin{itemize}
    \item{Approximability:} There exists constants $C>0$ and $k\geq 1$ such that for all $u \in \tilde{H}$,
    $
        \inf_{\bar{u}^h \in H^h} \|u-u^h\|_H \leq C h^k\|u\|_{\tilde{H}}. 
    $
    \item{Variational consistency:} The indicator functions $I_{H^h} \xrightarrow{\Gamma} I_H$ as $h\downarrow 0$ in the strong topology of $H$.
\end{itemize} 
\end{definition}
\end{framed}

\begin{remark}
   For conforming spaces, variational consistency is equivalent to strong Kuratowski convergence $H^h\xrightarrow{K} H$. Since $H^h \subset H$, the condition is satisfied if $\bigcup_{h>0} H^h$ is dense in $H$, ensuring that every $u\in H$ is the strong limit of a sequence $u^h \in H^h$.
   A standard example of admissible finite element spaces is given by conforming Lagrangian finite elements of polynomial degree $k$ on a family of shape‑regular, quasi‑uniform meshes.
\end{remark}

\begin{framed}
\begin{problem}[\Ptwoh~-- finite element approximation of \Ptwo]
Let $H^h$ be an admissible finite element space and let
\begin{equation*}
    V_\eta^h(u) := V_\eta(u)+I_{H^h}(u)  
\end{equation*}
be a finite element approximating regularized functional defined for $u \in H$. Find
\begin{equation*}
    u^h_\eta = \arg \left\{ \min_{\bar{u} \in H} V_\eta^h\left(\bar{u}\right)\right\} = \arg \left\{ \min_{\bar{u} \in H^h} V_\eta\left(\bar{u}\right)\right\}.
\end{equation*}
The minimizer $u^h_\eta\in H^h$ satisfies the discrete optimality conditions
\begin{equation*}
    a\left(u^h_\eta, v^h\right)+\eta \left(u^h_\eta, v^h\right)_L = \ell(v^h)~\forall~v^h \in H^h\cap K^\perp .
\end{equation*}
\end{problem}
\end{framed}

\begin{framed}
\begin{proposition}
Let $u$ be the solution to \Pone~and $u_\eta^h$ be the solution to \Ptwoh. Assume furthermore
$
\| u_\eta \|_{\tilde{H}} \leq C \| \ell \|_L
$
for some constant $C$. Then, the following error estimate holds
\begin{equation*}
\| u - u_\eta^h \|_H \leq C \left( \eta + h^k \right)\| \ell \|_L.
\end{equation*}
\end{proposition}
\end{framed}

\begin{proof}
By the triangle inequality, the total error is bounded as
$
\|u_0-u_\eta^h\|_H \leq \| u_0 -u_\eta\|_H +\|u_\eta-u_\eta^h \|_H
$. 
The first term is bounded by $\|u_0-u_\eta\|\leq C\eta\|\ell\|_{H'}$. The second term represents the finite element approximation error.
\end{proof}

\begin{framed}
\begin{proposition}
    \Ptwoh~$\Gamma$-converges to \Pone~as $\eta \downarrow 0$ and $h \downarrow 0$, simultaneously. In particular $u^h_\eta \to u_0 \in H$, where $u^h_\eta$ is the unique solution of $\mathrm{P2}^h$ and $u_0$ is the unique solution of \Pone.
\end{proposition}
\end{framed}

\begin{proof}
To begin, set $\eta = h = \epsilon \downarrow 0$ so that the convergence analysis can be carried out with respect to a single small parameter. Also let $\Pi^h = \Pi^\epsilon$ be the projection onto $H^h = H^\epsilon$. With this notation, the discrete regularized functional becomes
\begin{equation*}
    V_\epsilon(u):=V^{h =\epsilon}_{\eta=\epsilon}(u)=V(u)+\frac{\epsilon}{2}\|u\|_L+I_{H^\epsilon}(u)
\end{equation*}
and the corresponding minimizer $u_\epsilon: = u^{h=\epsilon}_{\eta=\epsilon}$ satisfies
\begin{equation*}
    u_\epsilon = \arg\left\lbrace \inf_{\bar{u}\in H} V_\epsilon(\bar{u})\right\rbrace.
\end{equation*}

\noindent \textit{i}) Equi-coercivity: For the regularized bilinear form $a_\epsilon(\cdot, \cdot):= a(\cdot,\cdot)+\epsilon(\cdot,\cdot)_L$, uniform coercivity holds on $K^\perp$. Since $H^\epsilon \cap K^\perp \subset H \cap K^\perp = K^\perp$, the family $\lbrace V_\epsilon\rbrace_{\epsilon\geq 0}$ is equi-coercive on $K^\perp$.

\noindent \textit{ii}) The \textit{liminf} inequality:
Let $u \in K^\perp$ and let $u_\epsilon \in H^\epsilon\cap K^\perp$ be such that $u_\epsilon \rightharpoonup u \in K^\perp$. Then,
\begin{equation*}
    V_\epsilon(u_\epsilon) = V(u_\epsilon)+\frac{\epsilon}{2}\|u_\epsilon\|_L^2+I_{H^\epsilon}(u_\epsilon)\geq V(u_\epsilon)
\end{equation*}
because $I_{H^\epsilon}(v^\epsilon) = 0$ for all $v^\epsilon \in H^\epsilon$.
Given the properties of $a(\cdot,\cdot)$ and $\ell(\cdot)$ defined on $K^\perp$, $V(\cdot) = \frac12 a(\cdot,\cdot)-\ell(\cdot)$ is weakly lower semicontinuous on $K^\perp$, we obtain
\begin{equation*}
    \liminf_{\epsilon \downarrow 0} V_\epsilon(u_\epsilon) \geq \liminf_{\epsilon \downarrow 0} V(u_\epsilon) \geq V(u),
\end{equation*}
whenever $u_\epsilon \rightharpoonup u \in K^\perp$.

\noindent \textit{iii}) The existence of a \textit{recovery} sequence:
Let $u \in K^\perp$ and define $u_\epsilon \in H^\epsilon\cap K^\perp$ as the unique solution of
\begin{equation*}
    a_\epsilon(u_\epsilon,v^\epsilon) = a(u,v^\epsilon)~\forall~v^\epsilon\in H^\epsilon\cap K^\perp.
\end{equation*}
Recall that $a_\epsilon(\cdot,\cdot)$ is coercive on $K^\perp$ with a coercivity constant $c_a$ and continuous on $H$ with a continuity constant $C_a$ and both constants are independent of $\epsilon$. Now, by taking
$v^\epsilon = u_\epsilon \in H^\epsilon\cap K^\perp$, we can infer the uniform bound $\| u_\epsilon \|_H \leq C_a/c_a \|u\|_H$. Since bounded sequences in $H$ admit weakly convergent subsequences, and since the limit is unique, the whole sequence satisfies  $u_\epsilon \rightharpoonup u \in K^\perp$.
To prove that $\lbrace u_\epsilon\rbrace_{\epsilon \geq 0}$ is a recovery sequence, choose $v^\epsilon=u_\epsilon-\Pi^\epsilon u \in H^\epsilon\cap K^\perp$ and obtain
\begin{equation*}
a(u_\epsilon-u, u_\epsilon -\Pi^\epsilon u) = -\epsilon(u_\epsilon, u_\epsilon-\Pi^\epsilon u)_L.
\end{equation*}
Expanding, invoking the coercivity of $a_\epsilon(\cdot,\cdot)$ on $K^\perp$ and the continuity of $a(\cdot,\cdot)$ on $H$, and applying the triangle and Young's inequalities,
\begin{equation*}
c_a\|u-u_\epsilon\|^2_H \leq C\left(\epsilon\|u\|^2_L+\|u-\Pi^\epsilon u\|^2_H\right),
\end{equation*}
where $C > 0$ is a constant independent of $\epsilon$. Taking limits
\begin{equation*}
\lim_{\epsilon \downarrow 0} c_a\|u-u_\epsilon\|^2_H \leq \lim_{\epsilon \downarrow 0} C\left( \epsilon\|u\|^2_L+\|u-\Pi^\epsilon u\|^2_H\right) = 0.
\end{equation*}
Hence $u_\epsilon \to u \in K^\perp$. Finally, we obtain
\begin{equation*}
    \lim_{\epsilon \downarrow 0} V_\epsilon(u_\epsilon) = \lim_{\epsilon \downarrow 0} \left(a(u_\epsilon, u_\epsilon)-\ell(u_\epsilon)+\frac{\epsilon}{2}\|u_\epsilon\|^2_L\right) = V(u),
\end{equation*}
because $u_\epsilon \to u \in K^\perp$.
Thus, $\lbrace u_\epsilon \rbrace_{\epsilon\geq 0}$ is a recovery sequence for $u$.
From \textit{ii}) and \textit{iii}), the $\Gamma$-convergence of $V_\epsilon$ as $\epsilon\downarrow 0$ is established. Finally, \textit{i}) and the $\Gamma$-convergence provide the convergence of the minimum problems and of their minimizers.
\end{proof}

\section{Acknowledgments}

CGG gratefully acknowledges the financial support from the European Research Council through the ERC Consolidator Grant ``DATA-DRIVEN OFFSHORE'' (Project ID 101083157). IR acknowledges the funding received from the project PID2021-128812OB-I00 from the Spanish Ministry of Science and Innovation.

\bibliographystyle{plain}
\bibliography{references} 

\end{document}